\documentclass[11pt]{amsart}

\usepackage{amssymb,amsfonts}
\usepackage{amsmath,amscd}
\usepackage[all,arc]{xy}
\usepackage{enumerate}
\usepackage{mathrsfs}
\usepackage[pdftex]{graphicx}
\usepackage[utf8]{inputenc}
\usepackage[T1]{fontenc}
\usepackage[usenames,dvipsnames]{color}
\usepackage[bookmarks=false]{hyperref}
\usepackage{dsfont}
\usepackage{tikz}
\usetikzlibrary{automata,positioning}
\usepackage{multirow}

\theoremstyle{plain}
\newtheorem{thm}{Theorem}[section]
\newtheorem{cor}[thm]{Corollary}
\newtheorem{prop}[thm]{Proposition}
\newtheorem{lem}[thm]{Lemma}

\newenvironment{customThm}[1]{\paragraph*{\textbf{Theorem #1.}}\itshape}{\par}

\newenvironment{customCor}[1]{\paragraph*{\textbf{Corollary #1.}}\itshape}{\par}

\theoremstyle{definition}
\newtheorem{defn}[thm]{Definition}

\theoremstyle{remark}
\newtheorem{rmk}[thm]{Remark}

\newcommand{\bbC}{\mathbb{C}} 

\newcommand{\bbH}{\mathbb{H}} 
\newcommand{\bbN}{\mathbb{N}} 

\newcommand{\bbP}{\mathbb{P}}

\newcommand{\bbR}{\mathbb{R}} 

\newcommand{\bbZ}{\mathbb{Z}} 
\newcommand{\PSL}{\text{PSL}}

\newcommand*{\defeq}{\mathrel{\vcenter{\baselineskip0.5ex \lineskiplimit0pt
			\hbox{\scriptsize.}\hbox{\scriptsize.}}}%
	=}

\makeatletter
\let\c@equation\c@thm
\makeatother
\numberwithin{equation}{section}

\title{Prime number theorems for Basmajian-type identities}
\author{Yan Mary He}
\address{Department of Mathematics\\
  University of Toronto\\
  Toronto, On M5S 2E4}
\email{yanmary.he@mail.utoronto.ca} 
\date{\today}

\begin{document}
	
\begin{abstract}
We obtain asymptotic counting results with error terms for complex orthospectrum for Schottky groups and orbit counting function for quadratic polynomials. Moreover, we prove equidistribution of holonomy associated to these dynamical systems. Our results are obtained by considering generalized $L$-functions coming from the Basmajian-type identities introduced by the author in \cite{He}. We study the associated summatory functions using tools from analytic number theory and Thermodynamic Formalism, namely the Perron's formula and a Dolgopyat-type estimate on the spectrum of transfer operators.
\end{abstract}

\maketitle

\section{Introduction}
The Prime Number Theorem proved without error term by Hadamard and de la Valleé-Poussin in 1896 and with (classical) error term by de la Valleé-Poussin in 1899 gives the asymptotic distribution of prime numbers. If we consider the Gaussian primes, the prime ideals in $\bbZ[i]$, the Prime Number Theorem implies further that the arguments of Gaussian primes equidistribute over the circle.

In the light of the analogy between primitive closed geodesics and prime numbers,  the Prime Geodesic Theorems have also been extensively studied over the past 50 years using methods of Selberg zeta function, Selberg trace formula and Ruelle dynamical zeta function. For compact hyperbolic surfaces, Huber \cite{Hu} \cite{Hu1} gave the asymptotic estimates with error terms for the number of closed primitive geodesics and geodesic arcs. Margulis \cite{Margulis} generalized the main terms to compact manifolds of variable negative curvature and Pollicott-Sharp gave the error estimates for compact surfaces of variable negative curvature in \cite{PolSharp2}. For further generalizations and development, see Sarnak's thesis \cite{Sar}, Guilleop\'e \cite{Gui}, Lally \cite{La} and Naud \cite{Naud}.

In this paper, we study asymptotic distribution of orbits in familiar one-dimensional complex dynamical systems, namely Schottky groups and quadratic polynomials. In particular, we consider the orbit of a geodesic $\ell$ connecting two points in the limit set under the iterated function system (IFS) generated by the inverse branches of the Bowen-Series map or a quadratic polynomial. For each map $f$ in the IFS, we associate a {\it complex length} $L(f)$ to the image $f(\ell)$ whose real part measures the distance between $\ell$ and $f(\ell)$(or the length of $f(\ell)$) and whose imaginary part measures the angle that $f$ rotates $\ell$. Our main goal is to obtain the asymptotic distribution of the cardinality of the set $S(x) = \{f ~|~ |L(f)| \ge 1/x \}$ and show that as $x \to \infty$, the angles of $f(\ell)$ for $f \in S(x)$ equidistribute over the circle.

Our results are obtained via a unified approach by studying the summatory functions associated to some (generalized) $L$-functions, whose special values are the right-hand-side series of Basmajian-type identities which were introduced by the author in \cite{He}. 

These results are of independent interest in hyperbolic geometry and complex dynamics. We discuss them individually in details.

\subsection{Schottky groups}
In \cite{Bas}, Basmajian proved the following identity for a compact hyperbolic surface $S$ with geodesic boundary:
\begin{equation}\label{bas_torus}
\text{length}(\partial S) = \sum_{\gamma} \log \coth^2 \left(\dfrac{\text{length}(\gamma)}{2}\right) = \sum_{\gamma} \log c_{\gamma}
\end{equation}
where the sum is taken over all orthogeodesics $\gamma$ in $S$ (i.e. properly immersed geodesic arcs perpendicular to $\partial S$ at both ends).

Algebraically, to every orthogeodesic $\gamma$, we can associate a unique word $w$ in the fundamental group $\pi_1S = F_n$ for some $n$ and rewrite the term $c_{\gamma}$ in terms of $w$ which we denote by $c_w$. Denote by $\mathscr{L}$ the set of words representing orthogeodesics. If we deform the hyperbolic structure of $S$, i.e. a discrete faithful representation $\rho_0 : F_n \to \PSL(2,\bbR)$ to a Schottky representation $\rho : F_n \to \PSL(2,\bbC)$, we would still have the {\it formal} series $\sum_{w \in \mathscr{L}} \log c_w$.

In the spirit of treating geodesics as primes, for any Schottky group, we consider the following generalized $L$-function
$$F(s, \chi)= \sum_{|w|=n} \chi\left(\frac{\log c_w}{|\log c_w|}\right) |\log c_w|^{s}$$
where $\chi: S^1 \to S^1$ is a unitary character given by $\chi(z) = z^m$ for some integer $m$. To an $L$-function, there is associated {\it summatory functions} $$M_m(x) \defeq \sum_{|\log c_w|^{-1}\le x} \left(\frac{\log c_w}{|\log c_w|}\right)^m.$$ In particular, if $m=0$, $$M_0(x) \defeq \sum_{|\log c_w|^{-1}\le x} 1 = \text{Card}\{w\in \mathscr{L} ~|~ |\log c_w | \ge 1/x\}.$$
Our main result is the following theorem regarding the asymptotic distribution of $M_m(x)$. 

\vspace{0.3cm}
\begin{customThm}{5.1}
Let $\delta$ be the Hausdorff dimension of the limit set.
\begin{enumerate}
	\item For $m=0$, there exists $C_1>0, 0< d_1 <\delta$ such that for any $x \ge 1$, $$M_0(x) = \text{Card }\{w ~|~ |\log c_w | \ge 1/x\} = C_1 x^{\delta}+ O(x^{d_1}).$$
	\item For $m \neq 0$ and non-Fuchsian Schottky groups, there exists $0< d_1 <\delta, 1<\beta<2$ such that for any $x \ge 1$, $$M_m(x) = O((|m|+1)^{\beta}x^{d_1}).$$
\end{enumerate}
\end{customThm}
\vspace{0.3cm}

The geometric interpretation of Theorem \ref{thm_error} establishes the asymptotic counting of complex orthospectrum and equidistribution of holonomy as stated in the following corollaries. 

Denote $N(x) \defeq \text{Card} \{w \in \mathscr{L}: Re(|\gamma|) \le x \}$ i.e. the number of orthogeodesics with real part of complex length bounded by $x$. 
\vspace{0.3cm}
\begin{customCor}{5.2 \& 5.3}
Let $\delta$ be the Hausdorff dimension of the limit set.
\begin{enumerate}
\item There exist constants $0< d_1 <\delta$ and $C_1 >0$ such that $$N(x) = C_1e^{\delta x} + O(e^{d_1 x})$$ as $x \to \infty$.
\item For any non-Fuchsian Schottky group, there exist $C>0$ and $0< d_1 <\delta$ such that for any $f \in C^2(S^1)$, we have $$\sum_{|\log c_w|^{-1}\le x} f\left(\frac{\log c_w}{|\log c_w|}\right) = Cx^{\delta}\int_0^1 f(e^{2\pi i t})dt + O(t^{d_1})$$ where the implied constant depends on the $C^2$-norm of $f$.
\end{enumerate}
\end{customCor}
\vspace{0.3cm}

The main term $N(x) \sim Ce^{\delta x}$ in part (1) has also appeared in Parkkonen-Paulin \cite{ParPa} and Pollicott \cite{Pol}, where it was proved by using Bowen-Margulis measures in \cite{ParPa} and Poincar\'e-type series and Theormodynamic Formalism in \cite{Pol}.

\subsection{Quadratic polynomials}
Consider quadratic polynomials $f_c(z) = z^2+c$ with $c$ lying outside of the Mandelbrot set $\mathcal{M}$. Let $T_1$ and $T_2$ be the two branches of $f_c^{-1}$ and $z_1$ be the fixed point of $T_1$. Recall that the Basmajian-type identity for $f_c$ is given by
\begin{equation}\label{eq_poly}
2z_1 = \sum_{w \in \{T_1,T_2\}^*} (-1)^\eta \Big( w(T_1(-z_1)) - w(T_2(-z_1)) \Big) = \sum_{w \in \{T_1,T_2\}^*} w(I)
\end{equation}
where $\eta$ is the number of $T_2$'s in the word $w$ \cite{He}.

For each complex parameter $c \notin \mathcal{M}$, we consider the following $L$-function
$$G(s,m) = \sum_{w \in \{T_1,T_2\}^*} \left(\frac{w(I)}{|w(I)|}\right)^m|w(I)|^s$$ for $s \in \bbC$ and $m \in \bbZ$ so that $G(1,1)$ equals (\ref{eq_poly}). Again, our main theorem is the following asymptotic formula for the summatory functions
$$P_m(x) \defeq \sum_{|w(I)|^{-1}\le x} \left(\dfrac{w(I)}{|w(I)|}\right)^m.$$ 

\vspace{0.3cm}
\begin{customThm}{6.11}
	Let $\delta$ be the Hausdorff dimension of the Julia set.
	\begin{enumerate}
		\item For $m=0$, there exists $C_2>0, 0< d_2 <\delta$ such that for any $x \ge 1$, $$P_0(x) = \text{Card }\{w ~|~ |\log c_w | \ge 1/x\} = C_2 x^{\delta}+ O(x^{d_2}).$$
		\item For $m \neq 0$, there exists $0< d_2 <\delta, 1<\beta<2$ such that for any $x \ge 1$, $$P_m(x) = O((|m|+1)^{\beta}x^{d_2}).$$
	\end{enumerate}
\end{customThm}
\vspace{0.3cm}

\begin{customCor}{6.12}
For any $c \notin \mathcal{M}$, there exist $C>0$ and $0< d_2 <\delta$ such that for any $f \in C^2(S^1)$, we have $$\sum_{|w(I)|^{-1}\le x} f\left(\frac{w(I)}{|w(I)|}\right) = Cx^{\delta}\int_0^1 f(e^{2\pi i t})dt + O(t^{d_2})$$ where the implied constant depends on the $C^2$-norm of $f$. Here $\delta$ is the Haudorff dimension of the Julia set.
\end{customCor}
\vspace{0.3cm}

Oh-Winter  \cite{Oh} and Naud \cite{Naud} have considered the asymptotic distribution of the number of primitive periodic orbits for hyperbolic rational maps and certain quadratic polynomials respectively. The relationship between our orbit counting function $P_0(x)$ and the number of primitive periodic orbits is analogous to the relationship between the hyperbolic circle problem and primitive closed geodesics in the context of Kleinian groups (see Theorem $1$ and Theorem $2$ of \cite{PolSharp2}).

\subsection{Strategy of proofs}
The proofs of Theorem \ref{thm_error} and Theorem \ref{thm_errorcx} use ideas and tools from analytic number theory and Thermodynamic Formalism.

In analytic number theory, for an $L$-function $F(s, \chi)$, the ($k^{th}$ order) Perron's formula relates the summatory functions $M_m(x)$to $F(s, \chi)$ as follows.
\begin{equation}\label{eq_integral}
M^k_m(x) = \frac{1}{2\pi i} \int_{c-i \infty}^{c+i \infty} F(s, \chi) \frac{x^{s+k}}{s(s+1)\cdots (s+k)}ds
\end{equation}
where $c > \delta$ and $M^k_m$ is the $k$-time integral of $M_m$ (i.e. integrate $M_m(y)$ $k$ times). Theorem \ref{thm_error} and Theorem \ref{thm_errorcx} then follow from evaluating this integral using complex analysis. 

We evaluate the integral (\ref{eq_integral}) by shifting the path of integration and applying the residue theorem. To this end, we need to understand the meromorphic domains, in particular pole-free regions, of $F(s, \chi)$ and $G(s, \chi)$ as well as obtain an upper bound for $|F(s, \chi)|$ and $|G(s, \chi)|$ on those regions. This is achieved by using spectral theory of transfer operators from Thermodynamic Formalism.

A key step to make use of Thermodynamic Formalism is to express $F(s, \chi)$ (resp. $G(s, \chi)$) as a geometric series of {\it twisted} transfer operators $\mathcal{L}_{-s\tau, m\theta}$ for properly chosen H\"older potentials $\tau$ and $\theta$ which relate the geometric quantities and the (symbolic) dynamics.

We use methods of Parry-Pollicott \cite{PP}, Pollicott-Sharp \cite{PolSharp}, \cite{PolSharp2}  and Dolgopyat \cite{Dol} to study spectral properties of $\mathcal{L}_{-s\tau, m\theta}$ and obtain pole-free regions and bounds for $F(s, \chi)$ (resp. $G(s, \chi)$). In particular, we show that the $L$-function is analytic on the half-plane $Re(s) > \delta-\varepsilon$ for some $\varepsilon>0$ except for a simiple pole at $s = \delta$ if the character is trivial. Furthermore, the absolute values of $F$ and $G$ are uniformly bounded by $|Im(s)|+|m|$ where $m$ is the integer given by the character $\chi(z) = z^m$.

\subsection{Organization of the paper}
The paper is organized as follows. In Section \ref{sec_thermo}, we review tools from Thermodynamic Formalism that we will need. In Section \ref{sec_Fs}, we construct, for any Schottky group, an $L$-function from which we obtain the series in Basmajian-type identities as a special value. We discuss its analytic properties in Section \ref{sec_anaFs} and use them to study the summatory functions in Section \ref{sec_count}. Section \ref{sec_quad} is devoted to the parallel results for Basmajian-type identities for quadratic polynomials.

\subsection{Acknowledgments}
I am grateful to Mark Pollicott and Steve Lalley for helpful conversations. I thank Peter Shalen for learning number theory together, John Friedlander for providing us excellent references and Giulio Tiozzo for discussing some technical details.

\section{Spectrum of transfer operators} \label{sec_thermo}
Let $A_{n \times n}$ be the adjacency matrix for a directed graph where $n$ is the number of vertices in the graph. We associate a one-sided {\it subshift of finite type} $\Sigma_A^+$ to $A$ as $$\Sigma_A^+ = \{ \underline{i} = (i_0, i_1, \cdots ) ~|~ i_j \in \{1, \cdots, n\},  A_{i_j,i_{j+1}} = 1 \}.$$
The {\it shift map} $\sigma: \Sigma_A^+ \to \Sigma_A^+$ is defined by $\sigma((i_0, i_1, i_2, \cdots)) = (i_1, i_2, \cdots)$.

Given $0 < \kappa < 1$, we define a metric on $\Sigma_A^+$ by $d(\underline{x},\underline{y}) = \kappa^n$, where $n = n(\underline{x},\underline{y})$ is the largest number so that the sequences $\underline{x}$ and $\underline{y}$ agree in the first $n$ terms. With this metric, $\Sigma_A^+$ is homeomorphic to a Cantor set.

A non-negative square matrix $A$ is {\it irreducible} if for each pair of indices $(i,j)$ there exists a positive integer $n$ such that $A^n_{i,j}>0$. $A$ is {\it aperiodic} if there exists a positive integer $n$ such that $A^n_{i,j}>0$ for all $(i,j)$. 

A subshift of finite type $(\Sigma_A^+, \sigma)$ is {\it (topologically) transitive} if there exists a dense orbit. $(\Sigma_A^+, \sigma)$ is {\it (topologically) mixing} if for each pair of non-empty open sets $U, V \subset \Sigma_A^+$, there exists a positive integer $n$ such that $\sigma^nU \cap V \neq \emptyset$. $(\Sigma_A^+, \sigma)$ is transitive if and only if $A$ is irreducible and mixing if and only if $A$ is aperiodic.

Consider $C^{\alpha}(\Sigma_A^+, \bbR)$ the Banach space of real-valued H\"older continuous functions on $\Sigma_A^+$. Let $\tau \in C^{\alpha}(\Sigma_A^+, \bbR)$. We say $\tau$ is {\it eventually positive} if there exist an $n \in \bbN$ such that $\tau^n = \sum_{i=1}^{n-1} \tau \circ \sigma^i$ is positive. Given $\tau \in C^{\alpha}(\Sigma_A^+, \bbR)$ and $\theta \in C^{\alpha}(\Sigma_A^+, \bbR/2\pi\bbZ)$, for $s = \sigma + it \in \bbC$ and $m \in \bbZ$, define the {\it twisted transfer operator} $\mathcal L_{-s\tau, m\theta} : C^{\alpha}(\Sigma_A^+, \bbC) \to C^{\alpha}(\Sigma_A^+, \bbC)$ by $$\mathcal L_{-s\tau, m\theta} w(x) = \sum_{\sigma(y)=x} e^{-s\tau(y)+im\theta(y)}w(y).$$ 
Note that if $m = 0$, $\mathcal L_{-s\tau, 0}$ is the usual transfer operator. The $n$th iterate $\mathcal L_{-s\tau, m\theta}^n$ has the form $$\mathcal L_{-s\tau, m\theta}^n w(x) = \sum_{\sigma^n(y)=x} e^{-s\tau^n(y)+im\theta^n(y)}w(y)$$ where $\tau^n(y) = \sum_{j=0}^{n-1} \tau(\sigma^j y)$ and $\theta^n(y) = \sum_{j=0}^{n-1} \theta(\sigma^j y)$.

In our applications, we wish to construct $L$-functions using twisted transfer operators $\mathcal{L}_{-s\tau ,m\theta}$ so that we can interpret the right hand side of the Basmajian-type identities as a special value. To this end, we construct H\"older potentials $\tau$ and $\theta$ in such a way so that they link the geometry and the symbolic dynamics. The spectral properties of $\mathcal{L}_{-s\tau ,m\theta}$ then give us desired analytic properties of our $L$-functions. 



For {\it mixing} subshift of finite type, we have the following Ruelle-Perron-Frobenius Theorem.
\begin{thm}[Complex Ruelle-Perron-Frobenius Theorem \cite{PP}] \label{thm_RPF}
	Let $(\Sigma_A, \sigma)$ be a mixing subshift of finite type. Then
	\begin{enumerate}
		\item If $s$ is real, then the operator $\mathcal{L}_{-s\tau}$ has a simple isolated maximal positive eigenvalue $\lambda_{-s\tau} = e^{P(-sr)}$ with an associated strictly positive eigenfunction $\psi \in C^{\alpha}(\Sigma_A, \bbR)$. There is a unique probability measure $\mu$ on $\Sigma_A$ such that $\mathcal{L}_{-s\tau}^*\mu = e^{P(-s\tau)}\mu$ and $\int \psi d\mu =1$. Furthermore, the rest of the spectrum is contained in a disk of radius strictly smaller than $e^{P(-s\tau)}$. There is exponential convergence $|\mathcal{L}_{-s\tau}^n \mathds{1} - \lambda_{-s\tau}^n| \le C\lambda_{-s\tau}^n\theta^n$, for some $C > 1$, $0<\theta <1$ and $n \ge 1$.
		\item If $s$ is complex, then $\rho(\mathcal{L}_{-s\tau})$ is less than or equal to $e^{P(-Re(s)\tau)}$. $\rho(\mathcal{L}_{-s\tau})$ is strictly less than $e^{P(-Re(s)\tau)}$ unless $Im(s)\tau = u \circ \sigma - u + \Psi + a$ where $u \in C^0(\Sigma_A,\bbR)$ and $\Psi \in C^0(\Sigma_A, 2\pi\bbZ)$ and $a$ is a constant. Furthermore, if such an identity does hold, then $\mathcal{L}_{-s\tau}$ has a simple maximal eigenvalue $e^{P(-Re(s)\tau)+ia}$ and the rest of the spectrum is contained in a disk of radius strictly less than $e^{P(-Re(s)r)}$.
		\item If $\tau$ is real-valued and is cohomologous to a negative function, then for $t\in \bbR$, the function $t \mapsto P(-t\tau)$ is real analytic and strictly decreasing. Furthermore, $P(-t\tau) \to -\infty (\infty)$ as $t \to \infty (-\infty)$. In particular, $P(-t\tau)$ has a unique zero.
	\end{enumerate}
\end{thm}

\subsection{Non-mixing subshift of finite type}
If $A$ is not aperiodic, then the subshift of finite type is not mixing and the Ruelle-Perron-Frobenius Theorem would not apply. However, Pollicott-Sharp have shown in \cite{PolSharp} that there is always a subshift of finite type $\Sigma_B$ coming from a submatrix $B$ of $A$ which carries the spectrum of the transfer operator. In our applications, it turns out to be straightforward to find such submatrix $B$ from $A$.

Note that the subshift of finite type $\Sigma_B \subset \Sigma_A$ and any $r \in C^{\alpha}(\Sigma_A,\bbC)$ gives rise to $r' \in C^{\alpha}(\Sigma_B,\bbC)$ by restricting $r$ to $\Sigma_B$. We define the transfer operators $\mathcal{L}_{-sr} : C^{\alpha}(\Sigma_A, \bbC) \to C^{\alpha}(\Sigma_A, \bbC)$ and $\mathcal{L}_{-sr'} : C^{\alpha}(\Sigma_B, \bbC) \to C^{\alpha}(\Sigma_B, \bbC)$ in the usual way.

\begin{lem}[\cite{PolSharp} Lemma 2] \label{lem_samespec}
	The operators $\mathcal{L}_{-sr}$ and $\mathcal{L}_{-sr'}$ are both quasi-compact. Moreover, they have the same spectra and satisfy the following properties.
	\begin{enumerate}
		\item  If $s$ is real-valued, then the operators $\mathcal{L}_{-sr}$ and $\mathcal{L}_{-sr'}$ have a simple isolated maximal positive eigenvalue $\lambda_{-sr} = \lambda_{-sr'} = e^{P(-sr)}=e^{P(-sr')}$.
		\item $\rho_e(\mathcal{L}_{-sr}) \le (1/2)^{\alpha}e^{P(-Re(s)r)}$ and the same is true for $\mathcal{L}_{-sr'}$.
	\end{enumerate}
\end{lem}

\subsection{Strong non-integrability and Dolgopyat-type estimate}
In this subsection, we give a Dolgopyat-type estimate on the spectrum of twisted transfer operators. We begin by defining a property called {\it strong non-integrability (SNI)} for the H\"older potential $\tau$. Recall that $\Sigma_A^-$ is the negative one-sided subshift of finite type, namely, $(\cdots, i_{-2}, i_{-1}, i_0) \in \Sigma_A^-$ if and only if $(i_0,i_1, i_2, \cdots) \in \Sigma_A^+$.

Given $\underline{i}, \underline{j} \in \Sigma^+_A$, $\zeta, \eta \in \Sigma_A^-$ with $A_{\zeta_{0},i_0} = A_{\zeta_{0},j_0} = A_{\eta_{0},i_0} = A_{\eta_{0},j_0} = 1$, the {\it temporal distance function} $\phi_{\zeta, \eta}$ is given by 
\begin{align*}
\phi_{\zeta, \eta}(\underline{i}, \underline{j}) = \sum_{k=0}^{\infty} & \tau(\zeta_{-k} \cdots \zeta_{0}\underline{i}) - \tau(\zeta_{-k} \cdots \zeta_{0}\underline{j})\\
& - \tau(\eta_{-k} \cdots \eta_{0}\underline{i}) + \tau(\eta_{-k} \cdots \eta_{0}\underline{j}).
\end{align*}	

Note that for each $k \ge 0$, $|\tau(\zeta_{-k} \cdots \zeta_{0}\underline{i}) - \tau(\zeta_{-k} \cdots \zeta_{0}\underline{j})| \le ||\tau|| \kappa^k d(\underline{i}, \underline{j})^{\alpha}$. Therefore, the series $\phi_{\zeta, \eta}(\underline{i}, \underline{j})$ converges exponentially fast and $|\phi_{\zeta, \eta}(\underline{i}, \underline{j})| \le \frac{2}{1-\kappa} ||\tau|| d(\underline{i}, \underline{j})^{\alpha}$.

\begin{defn}
	An eventually positive function $\tau \in C^{\alpha}(\Sigma^+_A)$ is {\it strongly non-integrable (SNI)} if there exist $N>0, \delta>0, \zeta, \eta \in \Sigma_A^-$ and $\underline{i}, \underline{j} \in \Sigma^+_A$ with $A_{\zeta_{0},i_0} = A_{\zeta_{0},j_0} = A_{\eta_{0},i_0} = A_{\eta_{0},j_0} = 1$ and $d(\underline{i}, \underline{j}) \le \kappa^N$ such that $$|\phi_{\zeta, \eta}(\underline{i}, \underline{j})| \ge \delta \kappa^{N\alpha}.$$
\end{defn}

\begin{rmk}
This is the same definition as the one given by Dolgopyat in \cite{Dol}. Indeed, setting $w^1 = \zeta \underline{i}, w^2 = \zeta \underline{j}, w^3 = \eta \underline{i}, w^4 = \eta \underline{j}$
\begin{align*}
\phi(w^1, w^2, w^3, w^4) & = \sum_{n=-\infty}^{\infty} \tau(\sigma^nw^1) - \tau(\sigma^nw^2) - \tau(\sigma^nw^3) + \tau(\sigma^nw^4)\\
& = \sum_{n=0}^{\infty} \tau(\sigma^{-n}w^1) - \tau(\sigma^{-n}w^2) - \tau(\sigma^{-n}w^3) + \tau(\sigma^{-n}w^4)\\
& = \phi_{\zeta, \eta}(\underline{i}, \underline{j}) 
\end{align*}
\end{rmk}

Let $\sigma_0$ be the root of $P(-\sigma\tau)=0$ where $P$ is the topological pressure. 

\begin{thm} \label{thm_Dolg}
Let $\tau \in C^{\alpha}(\Sigma_A, \bbR)$ be eventually positive and strongly non-integrable and let $\theta \in C^{\alpha}(\Sigma_A, \bbR/2\pi\bbZ)$. There exist $R>0, \epsilon>0, C>0, 1<\beta <2$ and $0<\rho<1$ such that if $|\sigma-\sigma_0|<\epsilon$ and $|t| > R$, then we have
$$||\mathcal{L}_{-s\tau, m\theta}^n || \le C(|t|+|m|)^{\beta}\rho^n.$$
\end{thm}

\begin{rmk}
This is essentially Lemma 4.2 of Dolgopyat \cite{Dol}. Here we twist the transfer operator by adding $m\theta$ to the imaginary part of the Holder potential $-s\tau$. Since $\tau$ and $\theta$ do not cancel each other out, the proof of the theorem follows from that of Lemma 4.2 given in \cite{Dol}. 
\end{rmk}

\section{$L$-functions for Schottky groups} \label{sec_Fs}
If $S$ is a compact hyperbolic surface with geodesic boundary, an {\it orthogeodesic} on $S$ is a properly immersed geodesic arc that is perpendicular to $\partial S$ at both ends. In \cite{Bas}, Basmajian 
proved the following identity:
\begin{equation}\label{bas_torus}
\text{length}(\partial S) = \sum_{\gamma} 2 \log \coth \left(\dfrac{\text{length}(\gamma)}{2}\right)
\end{equation}
where the sum is taken over all orthogeodesics $\gamma$ in $S$. A hyperbolic structure on $S$ is a discrete faithful representation $\rho_0: F_n \to \PSL(2,\bbR)$ where $F_n = \pi_1(S)$ is the rank $n$ free group. If we deform $\rho_0$  into a Schottky representation $\rho: F_n \to \PSL(2,\bbC)$, the right hand side series is still formally well-defined.

The main objective of this section is to construct, for any Schottky group, an $L$-function $F(s,m)$ so that we are able to interpret the right hand side of the Basmajian-type identity as a special value derived from $F(s,m)$. More specifically, we would like to use Thermodynamic Formalism to express $F(s,m)$ as a geometric series of {\it twisted} transfer operators $\mathcal{L}_{-s\tau, m\theta}$ for properly chosen H\"older potentials $\tau$ and $\theta$ in such a way that we can recover the formal absolute series by setting $s=m=1$. 

This is feasible because the set of orthogeodesics at a Fuchsian representation forms a {\it regular language} \cite{He}, which will lead us to the symbolic coding of the dynamical systems of orthogeodesics as a subshift of finite type, the first step towards applying Thermodynamic Formalism. We start by recalling these basic properties of orthogeodesics.

\subsection{Basics about orthogeodesics}
Let $S$ be a compact hyperbolic surface with non-empty geodesic boundary. An {\it orthogeodesic} on $S$ is a properly immersed geodesic arc that is perpendicular to the boundary $\partial S$ at both endpoints. Let $\alpha_1,\cdots, \alpha_k \in \pi_1S$ represent the free homotopy classes of boundary geodesics $a_1,\cdots, a_k$, respectively. Denote $H_j \defeq \langle \alpha_j \rangle$, the subgroup of $\pi_1S$ generated by $\alpha_j$. Then there is a bijection between the set of orthogeodesics on $S$ and the set of double cosets of the form $H_pwH_q$ where $w \in \pi_1S$ is not in $H_p \cap H_q$, for $p,q =1,\cdots,k$ (\cite{He}, Proposition 3.2). 

Let $\mathcal{A}$ be a symmetric generating set of $\pi_1S$. Choose an ordering on $\mathcal{A}$. This determines a unique reduced lexicographically first (RedLex) representative $w$ of each orthogeodesic $\gamma$. Let $\mathscr{L}_{p,q}$ be the set of nontrivial RedLex double coset representatives for fixed $p, q$. Then the set $$\mathscr{L} \defeq \coprod\limits_{1 \le p, q \le k} \mathscr{L}_{p,q}$$ is naturally in bijection with the set of orthogeodesics on $S$.

\begin{prop}[\cite{He}, Proposition 3.4]
	$\mathscr{L}$ is a {\it regular language} over the {\it alphabet} $\mathcal{A}$.
\end{prop}

Recall that for a fixed finite alphabet $\mathcal{A}$, a {\it language} is a subset of the set of all words on $\mathcal{A}$. A language is {\it regular} if it consists of exactly the words accepted by some finte state automaton. A {\it finite state automaton} (FSA) on a fixed alphabet $\mathcal{A}$ is a finite directed graph $G$ with a starting vertex $*$ and a subset of the vertices called the {\it accept states}, whose oriented edges are labeled by letters of $\mathcal{A}$ so that there is at most one outgoing edge with any given label at each vertex. A word is {\it accepted} by a finite state automaton if there is a path realizing the word which starts with $*$ and ends on an accept state. Figure \ref{fig:AugFSAquad} is an example of an FSA which accepts words in the alphabet $\{a,b \}$ ending with at least two $a$'s.
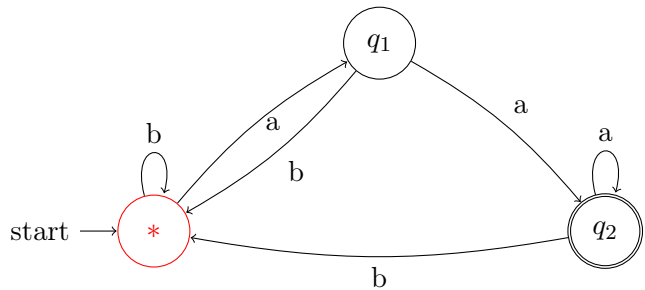
\begin{figure}[h!]
	\centering
	\begin{tikzpicture} [shorten >=0.3pt,node distance=0.4cm,auto] 
	\node[red][state, initial] (q0) at (0,0)  {$*$};
	\node[state] (q1) at (3,2.5) {$q_1$}; 
	\node[state,accepting] (q2) at (6,0) {$q_2$};
	\path[->] 
	(q0) edge [bend left=10, right] node [swap] {a} (q1)
	edge [loop above] node [swap] {b} (q0)
	(q1) edge [bend left=10] node {a} (q2)
	edge [bend left=10] node {b} (q0)
	(q2) edge [loop above] node [swap] {a} (q2) 
	edge [bend left=10] node {b} (q0);
	\end{tikzpicture}
	\caption{An FSA accepting all words over $a,b$ that end with at least two $a$'s. The vertex $q_2$ is the accept state.} 
	\label{fig:AugFSAquad}
\end{figure}

\subsection{Symbolic dynamics for orthogeodesics}
Given the directed graph (or FSA) parametrizing the set of orthogeodesics, we augment it by adding a vertex named $0$ with one incoming edge for each accept state and an edge from itself. Let $A$ be the adjacency matrix of this  {\it augmented} FSA. Then there is a one-sided subshift of finite type $\Sigma_A^+$ associated to the orthogeodesics. An orthogeodesic $\gamma$ is coded in the form of $i_0 i_1 \cdots i_{n} \dot 0$, which is a sequence ending with infinitely many $0$s. If $w$ is the unique RedLex representative of $\gamma$, then $w = l(i_0,i_1)\cdots l(i_{n-1},i_n)$ where $l(i_j,i_{j+1})$ is the label on the edge $(i_j,i_{j+1})$ in the augmented FSA.

Before defining a metric on $\Sigma_A^+$, we recall the following lemma due to S. Lalley.
\begin{lem}[\cite{La}] \label{lem_lalley}
	There exist constants $C < \infty$, $0 < \kappa < 1$ such that for each reduced word $w$ of length $n$, $|w(\Lambda_{\Gamma})| \le C\kappa^n$.
\end{lem}

Let $\kappa$ be as in the lemma above, we define a metric on $\Sigma_A^+$ by $$d(\underline{x},\underline{y}) = \kappa^n$$ where $n = n(\underline{x},\underline{y})$ is the largest number so that the sequences $\underline{x}$ and $\underline{y}$ agree in the first $n$ terms.

\subsection{Non-mixing subshift of finite type}
In general $A$ is not aperiodic as there are conditions on the two ends of every word in the regular language of orthogeodesics (coming from taking double quotient), as well as the added vertex $0$ in the augmented graph. If the surface $S$ has only one boundary component, then $\mathscr{L} = \mathscr{L}_{1,1}$. In this case, it is not hard to see that there is submatrix $B$ of $A$ that is aperiodic; let $B$ be the adjacency matrix of the subgraph obtained from the augmented FSA by deleting the vertex $0$, the vertices corresponding to the starting and ending conditions and the edges going in and out from those vertices. Thus, the subshift of finite type $\Sigma_B^+ \subset \Sigma_A^+$ is in one-to-one correspondence to the set of reduced words in the alphabet of a symmetric generating set. 

If $S$ has more than $k \ge 2$ boundary components, then in fact the language $\mathscr{L}_{p,q}$ is regular for each $p,q = 1, \cdots, k$. It follows that for each $p$ and $q$, we get the same subshift of finite type $\Sigma_B^+ \subset \Sigma_A^+$ which is in one-to-one correspondence to the set of reduced words in the alphabet of a symmetric generating set. Therefore, we will consider H\"older potentials $\tau' : \Sigma_B^+ \to \bbR$ and $\theta': \Sigma_B^+ \to \bbR/2\pi\bbZ$ in the following subsection which are restrictions to $\Sigma_B^+$ of H\"older potentials defined on $\Sigma_A^+$.

\subsection{Constructing $L$-functions}
We define H\"older potentials $\tau$ and $\theta$ whose associated transfer operator would give us the desired $L$-function.

Given a Schottky representation $\rho : F_k \to PSL_2(\bbC)$, denote by $\ell$ the axis (in the universal cover) of the boundary element. If $w$ is a word in the regular language, denote by $d_w $ the {\it complex} distance between $\ell$ and $w \cdot \ell$ in $\bbH^3$ and we have $$\coth^2 \left(\frac{d_w}{2}\right) = [\ell(-\infty),\ell(\infty); w\cdot \ell(-\infty), w\cdot \ell(\infty)] = c_w.$$

Then formally we have $\text{RHS} = \sum_{w } \log c_w$ where each term in the series is a complex number. 
\begin{prop}\label{prop_holder1}
	There exists a H\"older map $\tau: \Sigma_A^+ \to \mathbb R$ such that for any integer $n \ge 1$, we have $$ \tau^n(i_0 i_1 \cdots i_{n} \dot 0) =  -\log \left|\log c_w\right| $$
	where $\tau^n(x) = \sum_{j=0}^{n-1} \tau(\sigma^j x)$, $w = l(i_0,i_1)\cdots l(i_{n-1},i_n)$.
\end{prop}

\begin{proof}
	We first define a map $\tau: \Sigma_A^* \to \mathbb R$ on the set $$\Sigma_A^* = \{i_0 i_1 \cdots i_{n}\dot 0 ~|~ l(i_0,i_1)\cdots l(i_{n-1},i_n) \in \mathcal{L} \}$$ by
	$$\tau(i_0 i_1 \cdots i_{n} \dot 0) = \log \left|\log c_{\sigma (w)}\right|  -\log \left|\log c_w \right|$$
	for $|w| \ge 1$ and $\tau(\dot 0) = 0$. Here $\sigma(w) = l(i_1,i_2)\cdots l(i_{n-1},i_n)$. We assume the following lemma for now.
	
	\begin{lem}\label{lem_Lip}
		$\tau$ is H\"older continuous on $\Sigma_A^*$.
	\end{lem}
	
	Since $\Sigma_A^*$ is dense in $\Sigma_A^+$ and $\tau$ is uniformly continuous, $\tau$ extends to a H\"older map on the entire shift space $\Sigma_A^+$.
\end{proof}

\begin{proof}[Proof of Lemma \ref{lem_Lip}]
Since the displacement function $d_w = d(\ell, w\ell)$ is H\"older \cite{Pol}, we need to show that  $f(z) = \log |\log \coth(z/2)|$ is H\"older. First we note that there exists $c>0$ such that $Re|\gamma|>c$ for all orthogeodesics $\gamma$.

If $z \in \bbR$, $f(z) = \log \log \coth(z/2)$ is Lipschitz for $z \in (c, \infty)$ as $f'(z)$ is bounded on $(c, \infty)$.

If $z \in \bbC$, $\coth(z/2)$ does not approach $0$ or $\infty$ as $Re|\gamma|>c$ for some $c>0$. For a long word $w$, $w(0)$ and $w(\infty)$ can get close so that $w(0)/w(\infty) \to 1$. However, in this case, $z = 2\coth^{-1}(w(0)/w(\infty))$ has large real part and small imaginary part. Let $z_1 = x_1 + iy_1$ and $z_2 = x_2 + iy_2$ with $x_i$ large and $y_i$ small. Then $$|f(z_1) - f(z_2)| \le |f(x_1+iy_1) - f(x_2+iy_1)| +|f(x_2+iy_1) - f(x_2+iy_2)|$$ $$ \le C|x_1-x_2| + \varepsilon \le C|z_1-z_2|. $$
\end{proof}

\begin{prop}\label{prop_holder2}
	There exists a H\"older map $\theta: \Sigma_A^+ \to \mathbb R/2\pi\bbZ$ such that for any integer $n \ge 1$, we have $$ \theta^n(i_0 i_1 \cdots i_{n} \dot 0) =  \arg \left(\log c_w\right) $$
	where $\theta^n(x) = \sum_{j=0}^{n-1} \theta(\sigma^j x)$, $w = l(i_0,i_1)\cdots l(i_{n-1},i_n)$.
\end{prop}
\begin{proof}
Following the proof of Proposition \ref{prop_holder1}, we need to show that the map $\theta: \Sigma^*_A \to \mathbb R/2\pi\bbZ$ given by 
\begin{align*}
\theta(i_0 i_1 \cdots i_{n} \dot 0) &= \arg \left(\log \frac{w(0)}{w(\infty)}\right) - \arg \left(\log \frac{\sigma(w)(0)}{\sigma(w)(\infty)}\right)\\
\end{align*}
for $|w| \ge 1$ and $\theta(\dot 0) = 0$ is H\"older on $\Sigma_A^*$. 

Note that $\arg \left(\log \frac{x}{y}\right)$ is analytic in $(x,y)$ if $Re(\log \frac{x}{y})$ is positive. Let $w,u \in \Sigma_A^*$ with $d(w,u) = \kappa^n$. As elements in $\PSL(2,\bbC)$, $\rho(w)$ and $\rho(u)$ are orientation preserving so that the real part of $\log(w(0)/(w(\infty))$ is positive. Therefore, $|\arg \log\left(\frac{w(0)}{w(\infty)}\right)-\arg \log\left(\frac{u(0)}{u(\infty)}\right)| \le C\kappa^{n} R$ for some constant $C$ and $R$ is the diameter of the limit set. Similarly, $|\arg \log\left(\frac{\sigma(w)(0)}{\sigma(w)(\infty)}\right)-\arg \log\left(\frac{\sigma(u)(0)}{\sigma(u)(\infty)}\right)| \le C\kappa^{n-1} R$. Thus $\theta$ is H\"older since $\kappa^n+\kappa^{n-1} = O(\kappa^{n\alpha})$ for some $\alpha \in (0,1)$.
\end{proof}

\begin{lem}
$\tau$ is strongly non-integrable.
\end{lem}
\begin{proof}
For notational convenience, we set $\tau(w) = r(\sigma w) - r(w)$ where $r(w) = \log |\log (w(0)/w(\infty))|$. Suppose $\tau$ is not SNI. Then the temporal distance function
$$\phi_{\zeta, \eta}(\underline{i}, \underline{j}) = r(\zeta\underline{j}) - r(\zeta\underline{i}) + r(\eta\underline{i}) - r(\eta \underline{j})$$ is identically zero, i.e., 
\begin{align*}
r(\zeta\underline{j}) - r(\zeta\underline{i}) &= r(\eta\underline{j}) - r(\eta \underline{i})\\
\frac{|\log (\zeta\underline{j}(0)/\zeta\underline{j}(\infty))|}{|\log (\zeta\underline{i}(0)/\zeta\underline{i}(\infty))|} & =\frac{|\log (\eta\underline{j}(0)/\eta\underline{j}(\infty))|}{|\log (\eta\underline{i}(0)/\eta\underline{i}(\infty))|} 
\end{align*}
for any $\zeta, \eta, \underline{i}, \underline{j} \in \Sigma_A^+$.
This implies the derivative of the Bowen-Series map is constant on the Markov sets containing the limit set which is a contradiction.
\end{proof}

Given the H\"older potentials $\tau$ and $\theta$ constructed above, apply the iterates of $\mathcal L_{-s\tau, m\theta}: C^{\alpha}(\Sigma_A,\bbC) \to C^{\alpha}(\Sigma_A,\bbC)$ to the constant function $\mathds{1}$ at the point $\underline{x} = \dot 0$ and we obtain
\begin{align*}
\mathcal L_{-s\tau, m\theta}^n \mathds{1} (\dot 0)  &= \sum_{\sigma^n(y)=\dot 0} e^{-s\tau^n(y)+i m\theta^n(y)}= \sum_{\sigma^n(y)=\dot 0} \left(e^{-\tau^n( y)} \right)^s \cdot \left(e^{i\theta^n( y)} \right)^m\\
& =\sum_{|w|=n}\left(e^{\log |\log c_w| }\right)^s\left(\frac{\log c_w}{|\log c_w|}\right)^m \\
& = \sum_{|w|=n} \left(\frac{\log c_w}{|\log c_w|}\right)^m|\log c_w|^{s}.
\end{align*}

Therefore, we (formally) define a function $F(s,m)$ by
\begin{equation}
F(s,m) = \sum_{n=1}^\infty \mathcal L_{-s\tau, m\theta}^n\mathds{1}(\dot 0).
\end{equation}

This is an Hecke-type $L$-function of which the right-hand-side series of the Basmajian-type identity is a special value, namely $F(1,1)$.

\section{Analytic properties of $F(s,m)$} \label{sec_anaFs}
In this section, we discuss analytic properties of $F(s,m)$. In particular, we show that $F(s,m)$ is analytic on the half-plane $Re(s)>\delta$ and can be extended (meromorphically) to a strip $Re(s) \in (\delta-\varepsilon,\delta]$ with $|F(s,m)|$ being uniformly bounded by $|Im(s)|+|m|$. These properties will be used to study the summatory functions for counting in the next section. We start by determining the domain of convergence of $F_m(s)$.

\begin{thm}[Convergence theorem] \label{thm_conv}
Fix an integer $m$.	The complex function $F_m(s) = F(s,m)$ converges when $Re(s) > \delta$ and diverges when $Re(s) < \delta$, where $\delta$ is the Hausdorff dimension of the limit set of $\Gamma$.
\end{thm}
\begin{proof}
	Write $s = \sigma+it$. By Lemma \ref{lem_samespec}, $\mathcal L_{-\sigma \tau +i(-t\tau+m\theta)}$ and $\mathcal L_{-\sigma \tau' +i(-t\tau'+m\theta')}$ have the same spectra. By (2) of Theorem \ref{thm_RPF}, $$F_m(s) \asymp \sum_{n=1}^{\infty}\left(e^{P(-\sigma \tau')}\right)^n$$ which converges if and only if $P(-\sigma \tau')<0$. Since $P$ is monotone decreasing in $\sigma$ and $P$ has a unique zero at $\sigma_0$, $P(-\sigma \tau')<0$ if and only if $Re(s) > \sigma_0$. By Bowen's theorem \cite{Bowen} and the comparison calculation in \cite{He}, $\sigma_0 = \delta$.
\end{proof}

An immediate corollary of the proposition is an earlier result of the author.
\begin{cor} [\cite{He}]
	The right hand side series in Basmajian identity is absolutely convergent if and only if the Hausdorff dimension of the limit set is strictly smaller than $1$.
\end{cor}
\begin{proof}
	Apply the above proposition with $s=1, m=0$.
\end{proof}

\begin{prop} \label{lem_|F|}
	There exist $R>0, \epsilon>0$ and  $1<\beta <2$ such that if $|\sigma-\sigma_0|<\epsilon$ and $|t| > R$, 
	$$|F_m(s)| \le O((|t|+|m|)^{\beta}).$$
\end{prop}
\begin{proof}
	By Theorem \ref{thm_Dolg}, $||\mathcal{L}_{-s\tau, m\theta}^n \mathds{1}|| \le C(|t|+|m|)^{\beta}\rho^n$ for some $0<\rho<1$. Hence, $|F_m(s)| \le \dfrac{C}{1-\rho} (|t|+|m|)^{\beta}$.
\end{proof}

\begin{prop} \label{thm_simplepole}
Let $\varepsilon>0$ be as in the above proposition.
\begin{enumerate}
\item If $m=0$, the function $F_0(s)$ is analytic on the half-plane $Re(s) > \delta- \varepsilon$ except for a simple pole at $s = \delta$ with positive residue. It does not have any other poles on the line $Re(s) = \delta$.
\item If $m \neq 0$, the function $F_m(s)$ is analytic on the half-plane $Re(s) > \delta- \varepsilon$.
\end{enumerate}
\end{prop}

\begin{proof}
\begin{enumerate}
\item  The case $m=0$ is well-known and goes back to Parry-Pollicott \cite{PP} and Pollicott-Sharp \cite{PolSharp2}. For the convenience of the reader, we give a proof here.
We first prove that $s = \delta$ is a simple pole with positive residue. If $s \in \bbR$ and $s$ is in a small neighbourhood of $\delta$, then
\begin{align*}
F_0(s) &= \dfrac{\bbP_{\lambda}\mathcal L_{-sr} \mathds{1}(\dot 0)}{1-e^{P(-sr)}} + B(s).
\end{align*}
$\delta$ is a simple pole as $P(-sr)$ has a unique zero at $s= \delta$ by Theorem \ref{thm_RPF} part (3). The residue of $F_0(s)$ at $s=\delta$ is $$\displaystyle\lim_{s \to \delta}\dfrac{s-\delta}{1-e^{P(-sr)}} = \dfrac{1}{-P'(\delta r)}>0$$ as the pressure is decreasing with $s$.
	
Now we show $F_0(s)$ has no other poles on the line $Re(s) = \delta$. Suppose $F_0(s)$ has another pole at $s' = \delta+it$. If the spectral radius $\rho(\mathcal{L}_{-s'r})=1$, then by part (2) of Theorem \ref{thm_RPF}, 
	\begin{equation}\label{eq_a}
	Im(-s'r') = tr' = u \circ \sigma - u + \Psi +a
	\end{equation}
	for some $u \in C^0(\Sigma_B,\bbR), \Psi \in C^0(\Sigma_B, 2\pi\bbZ)$ and $a \in \bbR$ and $e^{ia}$ is the simple maximum eigenvalue for $\mathcal{L}_{-s'r'}$. By perturbation theory of linear operators, there exists a neighbourhood $U$ of $s'$
	such that $e^{ia}$ is still the simple maximum eigenvalue $e^{P(-sr)}$ for $\mathcal{L}_{-sr}, s \in U$. Then for $s \in U$, using the same projection idea as in the previous  proof,
	\begin{align*}
	F_0(s) &= \left(\sum_{n=1}^{\infty} \mathcal L_{-sr}^n \bbP_{\lambda max} \mathds{1} \right) (\dot 0) + \left(\sum_{n=1}^{\infty} \mathcal L_{-sr}^n Q' \mathds{1} \right) (\dot 0)\\
	&= \left(\sum_{n=0}^{\infty} e^{ian} \bbP_{\lambda max}\mathcal L_{-sr} \mathds{1} \right) (\dot 0) + \left(\sum_{n=1}^{\infty} \mathcal L_{-sr}^n Q' \mathds{1} \right) (\dot 0)\\
	&= \left((1-e^{ia})^{-1} \bbP_{\lambda max}\mathcal L_{-sr} \mathds{1} \right) (\dot 0) + B(s)
	\end{align*}
	where $Q':C^{\alpha}(\Sigma_A, \bbC) \to C^{\alpha}(\Sigma_A, \bbC)$ is the projection associated to the spectrum $|z| < \lambda max$.
	
	$B(s)$ is analytic as the spetral radius of $L_{-sr} Q'$ is strictly smaller than $1$. Hence, $F_0(s)$ has a pole at $s'$ if and only if $e^{ia}=1$, i.e. $a=0$. Hence, equation \ref{eq_a} becomes
	\begin{equation}\label{eq_a=0}
	tr' = u \circ \sigma - u + \Psi 
	\end{equation}
	
	If $\underline{x} \in \Sigma_B$ is a periodic point with $\sigma^n (\underline{x}) = \underline{x}$, $\underline{x}$ is identified via the symbolic coding with a reduced word $w$ in the fundamental group whose homotopy class contains a unique geodesic $\gamma$. On the other hand, for such points $\underline{x}$, $-tr'^n(\underline{x}) \in 2\pi\bbZ$ by equation (\ref{eq_a=0}).  This implies that the set $\{\log\coth(|\gamma|/2) : \gamma \text{ closed geodesic on } \bbH^3/F_n\}$ is discrete, which is a contradiction as the geodesic flow of $\bbH^3/F_n$ is ergodic.

\item The case $m \neq 0$: this is an immediate consequence of Proposition \ref{lem_|F|}.
\end{enumerate}
\end{proof}

\section{Counting complex orthospectrum with error and equidistribution of holonomy} \label{sec_count}
In this section, we adopt methods from analytic number theory to study the summatory function
$$M_m(x) \defeq \sum_{|\log c_w|^{-1}\le x} \left(\frac{\log c_w}{|\log c_w|}\right)^m$$ associated to the $L$-function $F(s,m) = \displaystyle\sum_w \left(\dfrac{\log c_w}{|\log c_w|}\right)^m |\log c_w |^s$, where $s \in \bbC$ and $m \in \bbZ$. 
Our main theorem is the following.

\begin{thm}\label{thm_error}
Let $\delta$ be the Hausdorff dimension of the limit set.
\begin{enumerate}
\item For $m=0$, there exists $C_1>0, 0< d_1 <\delta$ such that for any $x \ge 1$, $$M_0(x) = \text{Card }\{w ~|~ |\log c_w | \ge 1/x\} = C_1 x^{\delta}+ O(x^{d_1}).$$
\item For $m \neq 0$ and non-Fuchsian Schottky groups, there exists $0< d_1 <\delta, 1<\beta<2$ such that for any $x \ge 1$, $$M_m(x) = O((|m|+1)^{\beta}x^{d_1}).$$
\end{enumerate}
\end{thm}

Denote $N(x) \defeq \text{Card} \{\gamma : Re(|\gamma| )\le x \}$, where $|\gamma|$ denotes the {\it complex} length of the orthogeodesic $\gamma$. An immediate corollary of Theorem \ref{thm_error} (1) gives the asymptotic counting of complex orthospectrum. 
\begin{cor}[Counting complex orthospectrum]   \label{cor_count}
There exist constants $0< d_1 <\delta$ and $C_1 >0$ such that $$N(x) = C_1e^{\delta x} + O(e^{d_1 x})$$ as $x \to \infty$. Here $\delta$ is the Haudorff dimension of the limit set.
\end{cor}
\begin{proof}
	When $Re|\gamma|$ is large, $|\log\coth(|\gamma|/2)| \sim |e^{-|\gamma|}|$ . Replace $x$ in part (1) of Theorem \ref{thm_error} by $e^x$ since $|e^{-|\gamma|}| = e^{-Re(|\gamma|)} \ge 1/x$ is equivalent to $Re(|\gamma|) \le \log x$.
\end{proof}

The main term $N(x) \sim C_1e^{\delta x}$ has also appeared in \cite{ParPa} and \cite{Pol}, where it was proved by using equidistribution in \cite{ParPa} and Poincar\'e series and Theormodynamic Formalism in \cite{Pol}.

As a corollary of Theorem \ref{thm_error} (2), we obtain equidistribution of holonomy. 
\begin{cor}[Equidistribution of holonomy]
For any non-Fuchsian Schottky group, there exist $C>0$ and $0< d_1 <\delta$ such that for any $f \in C^2(S^1)$, we have $$\sum_{|\log c_w|^{-1}\le x} f\left(\frac{\log c_w}{|\log c_w|}\right) = Cx^{\delta}\int_0^1 f(e^{2\pi i t})dt + O(t^{d_1})$$ where the implied constant depends on the $C^2$-norm of $f$. Here $\delta$ is the Haudorff dimension of the limit set.
\end{cor}
\begin{proof}
Since $f \in C^2(S^1)$, $f$ admits the following Fourier expansion $f(e^{2\pi i\theta}) = \sum_{n=-\infty}^{\infty}a_n(e^{2\pi i\theta})^n$ where $a_0 = \int_{0}^{1} f(e^{2\pi it}) dt$ and $a_n = O(|n|^{-2})$. Hence, $\displaystyle\sum_{|\log c_w|^{-1}\le x} f\left(\frac{\log c_w}{|\log c_w|}\right) = Cx^{\delta}\int_0^1 f(e^{2\pi i t})dt + O(|f|_{C^2}t^{d_1}).$
\end{proof}

\subsection{Proof of Theorem \ref{thm_error}}
In order to prove Theorem \ref{thm_error}, we first prove an estimate for the quantity $\widehat{\widehat{M}}_m(x) = \int_1^x \int_1^y M_m(z) dz.$

\begin{prop}\label{thm_M1error}
We have
\begin{enumerate}
\item If $m = 0$, there exists a number $0< d_1 <\delta$ such that for any $x \ge 1$, $$\widehat{\widehat{M}}_m(x) = C \dfrac{x^{\delta+2}}{\delta(\delta+1)(\delta+2)}+ O(x^{d_1 +2}).$$
\item If $m \neq 0$, there exists a number $0< d_1 <\delta, 1<\beta<2$ such that for any $x \ge 1$, $$\widehat{\widehat{M}}_m(x) = O((|m|+1)^{\beta}x^{d_1 +2}).$$
\end{enumerate}

\end{prop}

We first show that Theorem \ref{thm_error} follows from Proposition \ref{thm_M1error}.
\begin{proof}[Proof of Theorem \ref{thm_error}]
(1) The case $m=0:$
Denote $\widehat{M}_0(y) = \int_1^y M_0(z) dz$. Let $\eta \in (0,1)$ be a constant. We have,
$$\widehat{\widehat{M}}_0(x) - \widehat{\widehat{M}}_0(\eta x) = \int_{\eta x}^{x} \widehat{M}_0(y)dy \le (1-\eta) x\widehat{M}_0(x).$$
By Proposition \ref{thm_M1error}, $\widehat{\widehat{M}}_0(x) - \widehat{\widehat{M}}_0(\eta x) = \frac{C(1-\eta^{\delta+2})}{\delta(\delta+1)(\delta+2)}x^{\delta+2}+ O(x^{d+2})$. Therefore, $\widehat{M}_0(x) \ge \frac{C(1-\eta^{\delta+2})}{\delta(\delta+1)(\delta+2)(1-\eta)}x^{\delta+1} + O(x^{d+1}).$

Similarly, let $\eta' \in (1,2)$ be a constant. Then $\widehat{M}_0(\eta'x) - \widehat{M}_0(x) \ge (\eta'-1) xM_0(x)$ and $\widehat{M}_0(x) \le \frac{C(\eta'^{\delta+2}-1)}{\delta(\delta+1)(\delta+2)(\eta'-1)}x^{\delta+1}+ O(x^{d+1})$.

Applying the argument again, we get $M_0(x) =Cx^{\delta} + O(x^{d})$ for some constant $C>0$.

(2) The case $m \neq 0$ follows by the same argument.
\end{proof}

The rest of the section is devoted to the proof of Proposition \ref{thm_M1error}. We begin with the following key lemma which gives
an integral formula relating $\widehat{\widehat{M}}_m(x)$ and $F_m(s)$. The proof then follows by evaluating the integral using complex analysis.
\begin{lem} \label{lem_integral}
	For any $c > \delta$ and any real number $x \ge 1$, $$\widehat{\widehat{M}}_m(x) = \frac{1}{2\pi i} \int_{c-i \infty}^{c+i \infty} F_m(s) \frac{x^{s+2}}{s(s+1)(s+2)}ds.$$
\end{lem}
\begin{proof}
	We first show that the integral is absolutely convergent.
	\begin{align*}
	\int_{c-i \infty}^{c+i \infty} |F_m(s)| \left| \frac{x^{s+2}}{s(s+1)(s+2)}\right|ds &\le  \int_{c-i \infty}^{c+i \infty} F_m(c)  \frac{x^{c+2}}{\left|s(s+1)\right|}ds\\
	&\le F(c) x^{c+2} \int_{-\infty}^{\infty}  \frac{1}{c^2+t^2}dt < \infty
	\end{align*}
	
	Now we check the formula. 
	\begin{align*}
	\frac{1}{2\pi i}\int_{c-i \infty}^{c+i \infty} \frac{F_m(s) \cdot x^{s+2}}{s(s+1)(s+2)} ds & =\frac{1}{2\pi i} \sum_w \int_{c-i \infty}^{c+i\infty} \left(\frac{\log c_w}{|\log c_w|}\right)^m |\log c_w |^s \frac{x^{s+2}}{s(s+1)(s+2)} ds\\
	&= \frac{1}{2\pi i} \sum_w \int_{c-i \infty}^{c+i \infty} \left(\frac{\log c_w}{|\log c_w|}\right)^m \frac{(|\log c_w |x)^{s}x^2}{s(s+1)(s+2)} ds\\
	& = \sum_{x|\log c_w |\ge 1} \left(\frac{\log c_w}{|\log c_w|}\right)^m \dfrac{x^2}{2}\left(1-\dfrac{1}{x|\log c_w |}\right)^2\\
	& = \widehat{\widehat{M}}_m(x)
	\end{align*}
	The last equality can be seen by straightforward calculation. The second last equality comes from the following lemma.
	\begin{lem}
		Let $c>0$. For all $y>0$, 
		$$\frac{1}{2\pi i} \int_{c-i \infty}^{c + i \infty} \frac{y^s}{s(s+1)(s+2)}ds = \begin{cases}
		\frac{1}{2}(1-\frac{1}{y})^2, \text{ if } y>1\\
		0,  \text{ if }0<y \le 1
		\end{cases}$$
	\end{lem}
	\begin{proof}
		Straightforward calculation using the residue theorem.
	\end{proof}
\end{proof}

\begin{proof}[Proof of Proposition \ref{thm_M1error} (1)]
Let $c \in (\delta , \delta + \epsilon)$ and $d \in (\delta - \varepsilon , \delta)$ where $\epsilon$ is given in Dolgopyat's estimate (Theorem \ref{thm_Dolg}). By Lemma \ref{lem_integral}, $$ I =\widehat{\widehat{M}}_0(x) = \displaystyle\int_{c-i\infty}^{c+i\infty} F_0(s) \dfrac{x^{s+2}}{s(s+1)(s+2)}ds.$$ We shift the path of integration by letting $P_1$ be the segment $[c-i\infty, c-iT]$, $P_2$ be the segment $[c+iT, c+i\infty]$, $P_3$ be the segment $[c-iT, d-iT]$, $P_4$ be the segment $[d+iT, c+iT]$ and $P_5$ be the segment $[d-iT, d+iT]$. 

Applying the residue theorem, we obtain
\begin{align*}
I = Res|_{ s=\delta} + \sum_{i=1}^{5}\displaystyle\int_{P_i} F(s) \dfrac{x^{s+2}}{s(s+1)(s+2)} ds =C \dfrac{x^{\delta+2}}{\delta(\delta+1)(\delta+2)}+ \sum_{i=1}^{5} I_i 
\end{align*}
where $C$ is the residue of $F(s)$ at $s = \delta$ and $I_i = \displaystyle\int_{P_i} F(s) \dfrac{x^{s+2}}{s(s+1)(s+2)} ds$.

For $i=2$ (and similarly $i=1$), $s = c+it$ with $t>T$.
\begin{align*}
	|I_2| &\le \displaystyle\int_T^{\infty} |F_0(s)| \left|\dfrac{x^{s+2}}{s(s+1)(s+2)}\right|ds\\
	& \sim \displaystyle\int_T^{\infty} t^{\beta} \cdot \dfrac{ex^{\delta+2}}{t^3} dt \text{ (by Theorem } \ref*{thm_Dolg})\\
	& \sim x^{\delta+2} \left(\dfrac{-T^{\beta-2}}{\beta-2}\right) \to 0 \text{ as } T \to \infty \text{ (since } 1<\beta<2).
\end{align*}

For $i=4$ (and similarly $i=3$), $s = \sigma + iT$ with $\sigma \in [d,c]$.
$$\left|\frac{x^{s+2}}{s(s+1)(s+2)}\right| \le \dfrac{x^{c+2}}{T^3} = \dfrac{ex^{\delta
		+2}}{T^3}$$

Therefore,
\begin{align*}
	|I_4| &\le \displaystyle\int_d^{c} |F_0(s)| \dfrac{x^{\delta
			+2}}{T^3}d\sigma\\
	& \le \dfrac{x^{\delta+2}}{T^3} \left(\int_d^{c}|F(\sigma+ iT)|d\sigma \right) \\
	& \le \dfrac{x^{\delta+2}}{T^3} \left(T^{\beta} (c-d)\right) \to 0 \text{ as } T \to \infty.
\end{align*}

For $i=5$, $s = d + it$ with $t \in [-T,T]$.
$$\left|\frac{x^{s+2}}{s(s+1)(s+2)}\right| \le \dfrac{x^{d+2}}{|s||s+1||s+2|} \sim x^{d+2}\text{min} \{1, t^{-3}\} \text{ and } |F(s)| \le |t|^{\beta}$$

Therefore,
\begin{align*}
	|I_5| &\sim \displaystyle\int_{-T}^{T} |t|^{\beta} x^{d+2}\text{min} \{1, t^{-3}\} dt \\
	&=\displaystyle\int_{-1}^{1} |t|^{\beta} x^{d+2} dt +2 \displaystyle\int_{1}^{T} \frac{t^{\beta}}{t^3} x^{d+2} dt\\
	& = O( x^{d+2})
\end{align*}

Hence, $$\widehat{\widehat{M}}_m(x) = C \dfrac{x^{\delta+2}}{\delta(\delta+1)(\delta+2)}+ O(x^{d+2}).$$
\end{proof}

\begin{proof}[Proof of Proposition \ref{thm_M1error} (2)]
Again by Lemma \ref{lem_integral}, we need to evaluate 
\begin{align*}
\widehat{\widehat{M}}_m(x) &= \displaystyle\int_{c-i\infty}^{c+i\infty} F_m(s) \dfrac{x^{s+2}}{s(s+1)(s+2)}ds = I_5\\
& =  \displaystyle\int_{-T}^{T} (|t|+|m|)^{\beta} x^{d+2}\text{min} \{1, t^{-3}\} dt \\
&=\displaystyle\int_{-1}^{1} (|t|+|m|)^{\beta} x^{d+2} dt +2 \displaystyle\int_{1}^{T} (1+|m|)^{\beta}\frac{t^{\beta}}{t^3} x^{d+2} dt\\
& = O( (|m|+1)^{\beta}x^{d+2})
\end{align*}
The second last equality holds since $(|t|+|m|)^{\beta} \le (1+|m|)^{\beta}|t|^{\beta}$ for $|t|\ge 1$.
\end{proof}

\section{Quadratic polynomials} \label{sec_quad}
In this section, we discuss parallel counting results for quadratic polynomials $f_c(z) = z^2+c$ with $c$ lying outside of the Mandelbrot set $\mathcal{M}$. Recall that, given a complex parameter $c \in \bbC \setminus \mathcal{M}$, the Basmajian-type identity is formally given by (cf. \cite{He}) 
\begin{equation} \label{eq_id}
z_1 - (-z_1) = \sum_{w \in \{T_1,T_2\}^*} (-1)^\eta \Big( w(T_1(-z_1)) - w(T_2(-z_1)) \Big) = w(I)
\end{equation}
where $T_1(z) = \sqrt{z-c}$ and $T_2(z) = -\sqrt{z-c}$ are the two branches of $f_c^{-1}$, $z_1$ is the fixed point of $T_1$ and $\eta$ is the number of $T_2$'s in the word $w$. For notational convenience, we will denote by $w(I)$ the term in the right hand side series corresponding to the word $w$.

We proceed along the same lines as in the case of Schottky groups, namely we consider the following $L$-function
$$G(s,m) = \sum_{w \in \{T_1,T_2\}^*} \left(\frac{w(I)}{|w(I)|}\right)^m|w(I)|^s$$ for $s \in \bbC$ and $m \in \bbZ$ so that the right-hand-side of the identity (\ref{eq_id}) is given by $G(1,1)$. Using Thermodynamic Formalism, we prove analytic properties of $G(s,m)$ which are used to study the summatory function $$P_m(x) \defeq \sum_{|w(I)|^{-1}\le x} \left(\dfrac{w(I)}{|w(I)|}\right)^m.$$ The main goal of this section is to prove an asymptotic formula for $P_m(x)$ (Theorem \ref{thm_errorcx}).

\subsection{Constructing $G(s,m)$}
For $c \in \bbC \setminus \mathcal{M}$, the quadratic polynomial $f_c(z) = z^2 + c$ is expanding on the Julia set $J_c$ and there exists a Markov partition $\{P_1, P_2\}$ with respect to which $f_c$ is a Markov map and satisfies $P_1 \subset f_c(P_1)$, $P_2 \subset f_c(P_1)$, $P_1 \subset f_c(P_2)$, $P_2 \subset f_c(P_2)$. 

The augmented finite state automaton parametrizing all the words in the alphabet $\{1,2\}$ is shown in Figure \ref{fig:AugFSAquad}.
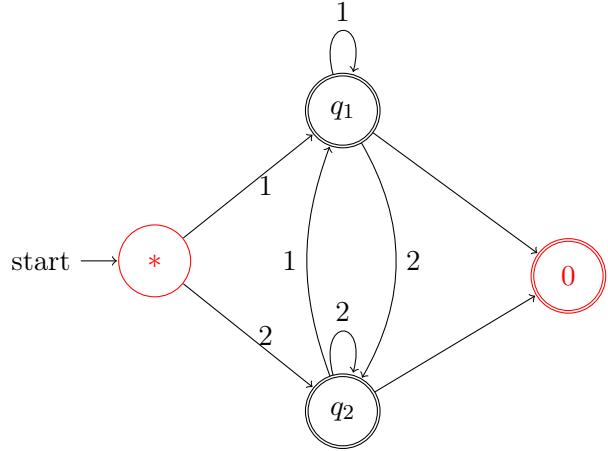
\begin{figure}[h!]
	\centering
	\begin{tikzpicture} [shorten >=0.5pt,node distance=0.7cm,auto] 
	\node[red][state, initial] (q0) at (0.5,1)  {$*$};
	\node[state,accepting] (q1) at (3,3) {$q_1$}; 
	\node[state,accepting] (q2) at (3,-1) {$q_2$};
	\node[red,state, accepting] (q9) at (6,0.8) {$0$};  
	\path[->] 
	(q0) edge [bend left=0, right] node [swap] {1} (q1)
	edge [bend left=0, right] node [swap] {2} (q2)
	(q1) edge [loop above] node [swap] {1} (q1) 
	edge [bend left=30] node {2} (q2)
	edge [above] node {} (q9)
	(q2) edge [loop above] node [swap] {2} (q2) 
	edge [bend left=20] node {1} (q1)
	edge [above] node {} (q9);
	\end{tikzpicture}
	\caption{Augmented FSA for quadratic polynomial} 
	\label{fig:AugFSAquad}
\end{figure}

Let $A$ be the adjacency matrix of the above directed graph. Note that $A$ is {\it not} aperiodic. However, there is a submatrix $B$ of $A$ which is aperiodic.
\[ A = \left[\begin{array}{cccc}
0 & 1 & 1 & 0\\
0 & 1 & 1 & 1\\
0 & 1 & 1 & 1\\
0 & 0 & 0 & 0\\ \end{array} \right] \text{ and }
B = \left[\begin{array}{cc}
1 & 1 \\
1 & 1 \\ \end{array} \right]\] 
Defines subshift of finite type $\Sigma_A$ and $\Sigma_B$.

The following lemma is a consequence of uniform hyperbolicity of $f_c$.
\begin{lem} \label{lem_lalleycx}
	There exist constants $C < \infty$, $0 < \kappa < 1$ such that for each word $w$ of length $n$, $|w(J_c)| \le C\kappa ^n$.
\end{lem}

Let $\theta$ be as in the lemma above, we define a metric on $\Sigma_A$ by $$d(\underline{x},\underline{y}) = \kappa ^n$$ where $n = n(\underline{x},\underline{y})$ is the largest number so that the sequences $\underline{x}$ and $\underline{y}$ agree in the first $n$ terms.

\begin{prop}
	There exists a H\"older map $\tau: \Sigma_A \to \mathbb R$ such that for any integer $n \ge 1$, we have $$ \tau^n(i_0 i_1 \cdots i_{n} \dot 0) = - \log |w(I)|$$
	where $\tau^n(x) = \sum_{j=0}^{n-1} \tau(\sigma^j x)$, $w = l(i_0,i_1)\cdots l(i_{n-1},i_n)$.
\end{prop}
\begin{proof}
	We first define a map $r: \Sigma_A^* \to \mathbb R$ on the set $$\Sigma_A^* = \{i_0 i_1 \cdots i_{n}\dot 0 ~|~ l(i_0,i_1)\cdots l(i_{n-1},i_n) \in \{T_1,T_2 \}^*\}$$ by
	$$\tau(i_0 i_1 \cdots i_{n} \dot 0) = \log |\sigma(w)(I)| - \log |w(I)|$$
	for $|w| \ge 1$ and $r(\dot 0) = 0$. We denote $\sigma(w) = l(i_1,i_2)\cdots l(i_{n-1},i_n)$. Note that $\tau^n(w\dot 0) =  -\log |w(I)|$. Moreover, $\tau$ is H\"older continuous on $\Sigma_A^*$ since $\tau(w\dot0) = \log \left|\frac{\sigma(w)(I)}{w(I)}\right| = \log |f_c'(z)|$ for some $z \in w(J_c)$ and $\log |f_c'|$ is analytic away from $0$. 
	Finally, $\tau$ extends to a H\"older map on the entire shift space $\Sigma_A$ as $\Sigma_A^*$ is dense in $\Sigma_A$ and $\tau$ is uniformly continuous.
\end{proof}

\begin{prop}
	There exists a H\"older map $\theta: \Sigma_A \to \mathbb R/2\pi\mathbb Z$ such that for any integer $n \ge 1$, we have $$ \theta^n(i_0 i_1 \cdots i_{n} \dot 0) = \frac{w(I)}{|w(I)|}$$
	where $\theta^n(x) = \sum_{j=0}^{n-1} \theta(\sigma^j x)$, $w = l(i_0,i_1)\cdots l(i_{n-1},i_n)$.
\end{prop}
\begin{proof}
	We first define a map $\theta: \Sigma_A^* \to \mathbb R$ on the set $$\Sigma_A^* = \{i_0 i_1 \cdots i_{n}\dot 0 ~|~ l(i_0,i_1)\cdots l(i_{n-1},i_n) \in \{T_1,T_2 \}^*\}$$ by
	$$\theta(i_0 i_1 \cdots i_{n} \dot 0) = \frac{w(I)}{|w(I)|} -\frac{\sigma(w)(I)}{|\sigma(w)(I)|}$$
	for $|w| \ge 1$ and $\theta(\dot 0) = 0$. We denote $\sigma(w) = l(i_1,i_2)\cdots l(i_{n-1},i_n)$. Moreover, $\theta$ is H\"older continuous on $\Sigma_A^*$ since $\theta(w\dot0) =  arg(f'(z))$ for some $z \in w(J_c)$ and $arg(f'(z))$ is analytic. By density of $\Sigma_A^*$ and uniform continuity of $\theta$, $\theta$ extends to a H\"older map on the entire shift space $\Sigma_A$.
\end{proof}

\begin{lem}
$\tau$ is strongly non-integrable.
\end{lem}
\begin{proof}
For notational convenience, we set $\tau(w) = r(\sigma w) - r(w)$ where $r(w) = \log |w(I)|$. Suppose $\tau$ is not SNI. Then the temporal distance function $\phi_{\zeta, \eta}(\underline{i}, \underline{j}) = r(\zeta\underline{j}) - r(\zeta\underline{i}) + r(\eta\underline{i}) - r(\eta \underline{j})$ is identically zero, i.e., $$r(\zeta\underline{j}) - r(\zeta\underline{i}) = r(\eta\underline{j}) - r(\eta \underline{i})$$ for any $\zeta, \eta, \underline{i}, \underline{j} \in \Sigma_A^+$.

Let $\zeta = \dot 0, \eta = 1\dot0, \underline{i} = 1\dot0$ and $\underline{j} = 2\dot0$. Then
$r(2) - r(1) = r(12) - r(11)$, i.e. $\dfrac{|2(I)|}{|1(I)|} = \dfrac{|12(I)|}{|11(I)|} = \dfrac{|T'_1(z)||2(I)|}{|T'_1(v)||1(I)|}$ for some $z \in T_2(I)$ and $v \in T_1(I)$. This implies that $|T'_1(z)| = |T'_1(v)|$ which is a contradiction as $|\sqrt{z-c}| \neq |\sqrt{v-c}|$ for $z \in T_2(I)$ and $v \in T_1(I)$.
\end{proof}

Given the H\"older potentials $\tau$ and $\theta$ constructed above, for $s = \sigma +it \in \bbC$ and $m \in \bbZ$, we apply $\mathcal L^n_{-s\tau, m\theta}$ to the constant function $\mathds{1}$ at the point $x = \dot 0$ and we obtain
\begin{align*}
\mathcal L_{-sr,m\theta}^n \mathds{1} (\dot 0)  &= \sum_{\sigma^n(y)=\dot 0} e^{-sr^n(y)+im\theta^n(y)}= \sum_{\sigma^n(y)=\dot 0} \left(e^{-r^n( y)} \right)^s (e^{i\theta^n(y)})^m\\
& =\sum_{|w|=n}\left(e^{\log |w(I)|}\right)^s \left(\frac{w(I)}{|w(I)|}\right)^m = \sum_{|w|=n} \left(\frac{w(I)}{|w(I)|}\right)^m |w(I)|^{s}.
\end{align*}

Therefore, for each $c \notin \mathcal{M}$, we (formally) define an $L$-function $G(s,m)$ by
\begin{equation}
G(s,m) = \sum_{n=1}^\infty \mathcal L_{-sr,m\theta}^n\mathds{1}(\dot 0).
\end{equation}

\subsection{Analytic properties of $G(s,m)$}
The spectral properties of transfer operators allow us to obtain the same analytic properties of $G(s,m)$ as those of $F(s,m)$ given in Section \ref{sec_anaFs}.

\begin{thm}[Domain of convergence] \label{thm_conv2}
	For each $m \in \bbZ$, $G(s,m)$ converges absolutely if and only if $Re(s) > \delta$, where $\delta$ is the Hausdorff dimension of the Julia set of $f_c$.
\end{thm}

Again, an earlier result of the author follows immediately from the proposition.
\begin{cor} [\cite{He}]
	The right hand side series in Basmajian-type identity is absolutely convergent if and only if the Hausdorff dimension of the Julia set is strictly smaller than $1$.
\end{cor}
\begin{proof}
	Apply the above theorem with $s=1, m=0$.
\end{proof}

\begin{prop} \label{lem_|G|}
	There exist $R>0, \epsilon>0$ and  $1<\beta <2$ such that if $|\sigma-\sigma_0|<\epsilon$ and $|t| > R$, 
	$$|G(s,m)| \le O((|t|+|m|)^{\beta}).$$
\end{prop}

\begin{prop} \label{thm_simplepole}
	Let $\varepsilon>0$ be as in the above proposition.
	\begin{enumerate}
		\item If $m=0$, the function $G_0(s)$ is analytic on the half-plane $Re(s) > \delta- \varepsilon$ except for a simple pole at $s = \delta$ with positive residue. It does not have any other poles on the line $Re(s) = \delta$.
		\item If $m \neq 0$, the function $G_m(s)$ is analytic on the half-plane $Re(s) > \delta- \varepsilon$.
	\end{enumerate}
\end{prop}

\subsection{Orbit counting in complex dynamics}
\begin{thm}\label{thm_errorcx}
	Let $\delta$ be the Hausdorff dimension of the Julia set.
	\begin{enumerate}
		\item For $m=0$, there exists $C_2>0, 0< d_2 <\delta$ such that for any $x \ge 1$, $$P_0(x) = \text{Card }\{w ~|~ |\log c_w | \ge 1/x\} = C_2 x^{\delta}+ O(x^{d_2}).$$
		\item For $m \neq 0$, there exists $0< d_2 <\delta, 1<\beta<2$ such that for any $x \ge 1$, $$P_m(x) = O((|m|+1)^{\beta}x^{d_2}).$$
	\end{enumerate}
\end{thm}

Analogously, the following corollary of Theorem \ref{thm_errorcx} (2) states the equidistribution of ``holonomy'' associated to the dynamical system $(J(f),f)$.
\begin{cor}
For any $c \notin \mathcal{M}$, there exist $C>0$ and $0< d_2 <\delta$ such that for any $f \in C^2(S^1)$, we have $$\sum_{|w(I)|^{-1}\le x} f\left(\frac{w(I)}{|w(I)|}\right) = Cx^{\delta}\int_0^1 f(e^{2\pi i t})dt + O(t^{d_2})$$ where the implied constant depends on the $C^2$-norm of $f$. Here $\delta$ is the Haudorff dimension of the Julia set.
\end{cor}
Similar to the case of Schottky groups, Theorem \ref{thm_errorcx} follows from the following proposition which gives an estimate for the quantity $\widehat{\widehat{P}}_m(x) = \int_1^x \int_1^y P(z) dz.$ 

\begin{prop}\label{thm_P1error}
We have
\begin{enumerate}
	\item If $m = 0$, there exists a number $0< d_2 <\delta$ such that for any $x \ge 1$, $$\widehat{\widehat{P}}_m(x) = C \dfrac{x^{\delta+2}}{\delta(\delta+1)(\delta+2)}+ O(x^{d_2 +2}).$$
	\item If $m \neq 0$, there exists a number $0< d_2 <\delta, 1<\beta<2$ such that for any $x \ge 1$, $$\widehat{\widehat{P}}_m(x) = O((|m|+1)^{\beta}x^{d_2 +2}).$$
\end{enumerate}
\end{prop}

Proposition \ref{thm_P1error} is proved the same way as Proposition \ref{thm_M1error} by applying the residue theorem to the integral $$\widehat{\widehat{P}}_m(x) = \frac{1}{2\pi i} \int_{c-i \infty}^{c+i \infty} G_m(s) \frac{x^{s+2}}{s(s+1)(s+2)}ds$$
where $c > \delta$ and $x \ge 1$.

\end{document}